\documentclass[11pt]{amsart}
\usepackage{amssymb, amsmath, amsthm, amscd}
\newtheorem{proposition}{Proposition}[section]

\newtheorem{corollary}[proposition]{Corollary}
\newtheorem{theorem}[proposition]{Theorem}

\theoremstyle{definition}

\newtheorem{example}[proposition]{Example}

\newtheorem{remark}[proposition]{Remark}

\newcommand{\thlabel}[1]{\label{th:#1}}
\newcommand{\thref}[1]{Theorem~\ref{th:#1}}
\newcommand{\selabel}[1]{\label{se:#1}}

\def\ot{\otimes}

\newcommand{\Cc}{\mathcal{C}}

\def\*C{{}^*\hspace*{-1pt}{\Cc}}
\def\text#1{{\rm {\rm #1}}}

\input xy
\xyoption {all} \CompileMatrices

\begin{document}

\title[Monomorphisms of coalgebras]{Monomorphisms of coalgebras}

\author{A.L. Agore  (Bucharest)} \thanks{The author was supported by
CNCSIS grant 24/28.09.07 of PN II "Groups, quantum groups, corings
and representation theory". }
\address{Faculty of Mathematics and Computer Science, University of Bucharest,
Str. Academiei 14, RO-010014 Bucharest 1, Romania \textbf{and}
\newline Department of Mathematics, Academy of Economic Studies,
Piata Romana 6, RO-010374 Bucharest 1,  Romania}
\email{ana.agore@fmi.unibuc.ro}



\begin{abstract}
We prove new necessary and sufficient conditions for a morphism of
coalgebras to be a monomorphism, different from the ones already
available in the literature. More precisely, $\varphi: C
\rightarrow D$ is a monomorphism of coalgebras if and only if the
first cohomology groups of the coalgebras $C$ and $D$ coincide if
and only if $\sum_{i \in I}\varepsilon(a^{i})b^{i} = \sum_{i \in
I} a^{i} \varepsilon(b^{i})$, for all $\sum_{i \in I}a^{i} \otimes
b^{i} \in C \square_{D} C$. In particular, necessary and
sufficient conditions for a Hopf algebra map to be a monomorphism
are given.
\end{abstract}

\maketitle

\section*{Introduction}

In any concrete category $\mathcal{C}$ the natural problem of
whether epimorphisms are surjective maps arise, as well as the
dual problem of whether the monomorphisms are injective maps. This
type of problems have already been studied before in several well
known categories: for example in \cite{R} it is shown that the
property of epimorphisms of being surjective holds in the
categories of von Neumann algebras, $C^{*}$-algebras, groups,
finite groups, Lie algebras, compact groups, while it fails to be
true in the categories of finite dimensional Lie algebras,
semisimple finite dimensional Lie algebras, locally compact groups
and unitary rings (see \cite{KNI}, \cite{Silv}). The more recent
paper \cite{C} deals with the same problems in the context of Hopf
algebras: several examples of non-injective monomorphisms and
non-surjective epimorphisms are given. It turns out that the above
problem is also intimately related to Kaplansky's first conjecture
in the way that every non-surjective epimorphism of Hopf algebras
provides a counterexample to Kaplansky's problem.

In \cite{N} the problems mentioned above are studied in the
category of coalgebras. The problem of whether epimorphisms of
coalgebras are surjective maps is easily dealt with using the
existence of a cofree coalgebra on every vector space: hence it is
enough to give a positive answer to this problem. The dual
problem, on the other hand, is more interesting: an example of a
non-injective monomorphism is given and several characterizations
of monomorphisms are proved (\cite[Theorem 3.5]{N}). In this note
we complete the above characterization with two new equivalences.
Our interest in this problem comes also from the fact that in the
light of \cite[Proposition 2.5]{C} which states that a morphism of
Hopf algebras is a monomorphism if and only if it is a
monomorphism viewed as a morphism of coalgebras it turns out that
the same characterization holds for Hopf algebra monomorphisms.\\
A detailed discussion regarding the theory of coalgebras can be
found in \cite{DNR} and\cite{Chin}.

\section{Preliminaries}\selabel{1}
Throughout this paper, $k$ is an arbitrary field. Unless specified
otherwise, all vector spaces, homomorphisms, algebras, coalgebras,
tensor products, comodules and so on are over $k$. For the
standard categories we use the following notations: ${}_k{\mathcal
{M}}$ ($k$-vector spaces),
 $k$-Coalg (coalgebras over $k$), ${\mathcal
{M}}^{C}$ (right $C$-comodules), ${}^{C}{\mathcal
{M}}^{D}$ ($(C,D)$-bicomodules).\\
For a coalgebra $C$, we use Sweedler's $\Sigma$-notation, that is,
$\Delta(c)= c_{(1)}\ot c_{(2)},~ (I\ot\Delta)\Delta(c)= c_{(1)}\ot
c_{(2)}\ot c_{(3)}$, etc. We also use the Sweedler notation for
left and right $C$-comodules: $\rho_M^{r}(m)= m_{[0]}\otimes
m_{[1]}$ for any $m\in M$ if $(M,\rho_M^{r})$ is a right $
C$-comodule and $\rho_N^{l}(n)= n_{<-1>}\otimes n_{<0>}$ for any
$n \in N$ if $(N,\rho_N^{l})$ is a left $C$-comodule. For further
details regarding the theory of comodules we refer to \cite{BW}.

If $M$ is a right $C$-comodule with structure map $\rho_{M}^{r}$
and $N$ a left $C$-comodule with structure map $\rho_{N}^{l}$, the
\textit{cotensor product} $M \square_{C} N$ is the kernel of the
$k$-linear map: $$\rho_{M}^{r} \otimes I - I \otimes \rho_{N}^{l}:
M \otimes N \rightarrow M \otimes C \otimes N.$$ Given comodule
maps $f:M \rightarrow M'$ and $g:N \rightarrow N'$, the $k$-linear
map $f \otimes g: M \otimes N \rightarrow M' \otimes N'$ induces a
$k$-linear map $f \square_{C} g: M \square_{C} N \rightarrow M'
\square_{C} N'$. A left $C$-comodule $M$ induces a functor $ -
\square_{C} M : {\mathcal {M}}^{C} \to {}_k{\mathcal {M}}$.

If $\varphi: C \rightarrow D$ is a coalgebra map then every right
(left) $C$-comodule $(M, \rho_{M}^{r})$ can be made into a right
(left) $D$-comodule with structure map $\tau_{M}^{r}: M
\rightarrow M \otimes D$ given by $\tau_{M}^{r}(m) = m_{[0]}
\otimes \varphi(m_{[1]})$. This association defines a functor
$\varphi^{*}_{C,D} : {\mathcal {M}}^{C} \rightarrow
{\mathcal{M}}^{D}$ usually called the \textit{corestriction functor}.\\
If $M \in {}^{C}{\mathcal {M}}^{D}$ we obtain a functor $ -
\square_{C} M: {\mathcal {M}}^{C} \rightarrow {\mathcal {M}}^{D}$.
In particular, $C$ becomes a left $D$-comodule via $\varphi$ and
we obtain a functor $- \square_{D} C: {\mathcal {M}}^{D} \to
{\mathcal {M}}^{C}$ called the \textit{coinduction functor} which
is a right adjoint to the corestriction functor
$\varphi^{*}_{C,D}$(\cite[22.12]{BW}). Recall that if $(F,G)$ is a
pair of adjoint functors, with $F: \mathcal{C} \rightarrow
\mathcal{D}$ and $G: \mathcal{D} \rightarrow \mathcal{C}$ if and
only if there exist two natural transformations $\eta:1_{\mathcal
{C}} \rightarrow GF$ and $\varepsilon: FG \rightarrow 1_{\mathcal
{D}}$, called \textit{the unit} and \textit{the counit} of the
adjunction, such that $G(\varepsilon_{D}) \circ \eta_{G(D)} =
I_{G(D)}$ and $\varepsilon_{F(C)} \circ F(\eta_{C}) = I_{F(C)}$
for all $C \in \mathcal{C}$ and $D \in
\mathcal{D}$.\\
Recall from \cite{IDN} the construction of the \textit{trivial
coextension} of a coalgebra $C$ by a $(C,C)$-bicomodule $N$. We
define a comultiplication and a counit on the space $C\rtimes N :=
C \oplus N$ given by:
\begin{eqnarray*}
\Delta (c,n) &{=}& (c_{(1)},0) \otimes (c_{(2)},0) + (n_{<-1>},0)
\otimes (0,n_{<0>}) + (0,n_{[0]}) \otimes (n_{[1]},0)\\
\varepsilon(c,n) &{=}& \varepsilon(c)
\end{eqnarray*}
for all $(c, n) \in C \oplus N$. In this way $C\rtimes N$ becomes
a coalgebra, known as the trivial coextension of $C$ and $N$.

Define the map $\pi_{C}: C\rtimes N \rightarrow C$ by $
\pi_{C}(c,n) = c$ for all $(c,n) \in C \oplus N$. For every $(c,n)
\in C\rtimes N$ we have:
 $$
\varepsilon_{C} \circ \pi_{C}(c,n) =\varepsilon_{C}(c) =
\varepsilon(c,n)
 $$
 and
\begin{eqnarray*}
(\pi_{C} \otimes \pi_{C})(\Delta(c,n)) &{=}& (\pi_{C} \otimes
\pi_{C})\bigl((c_{(1)},0) \otimes (c_{(2)},0) + (n_{<-1>},0)
\otimes (0,n_{<0>}) +\\
&&+ (0,n_{[0]}) \otimes (n_{[1]},0)\bigl)\\
&{=}& c_{(1)} \otimes c_{(2)} + n_{<-1>} \otimes 0 + 0 \otimes
n_{[0]}\\
&{=}& c_{(1)} \otimes c_{(2)}\\
&{=}& \Delta_{C} \circ \pi_{C}(c,n)
\end{eqnarray*}
Thus $\pi_{C}$ is a coalgebra map.

We also require some notions related to homological coalgebra. For
basic definitions and properties we refer to \cite{D}. We just
recall here, for further reference, the description of the first
cohomology group of a coalgebra $C$ with coefficients in a
$(C,C)$-bicomodule $N$:
\begin{eqnarray*}
H^{0}(N,C) &=& \{ \gamma \in N^{*} | (I \otimes
\gamma)\rho_{N}^{l} = (\gamma \otimes I) \rho_{N}^{r} \} \\
&=&  \{ \gamma \in N^{*} | n_{<-1>} \gamma (n_{<0>})  = \gamma
(n_{[0]}) n_{[1]}, \forall n\in N \}
\end{eqnarray*}

\section{Characterizations of monomorphisms of coalgebras}\selabel{2}

As mentioned before, a characterization of monomorphisms in
$k$-Coalg is given in \cite[Theorem 3.5]{N} and the equivalences
$(1) - (5)$ in \thref{c} are proved there. In what follows we
complete the description of monomorphisms in $k$-Coalg with two
more characterizations: the first one indicates a cohomological
description of monomorphisms while the other is an elementary one
involving the cotensor product $C \square_{D} C$.

\begin{theorem}\thlabel{c}
Let $\varphi: C \rightarrow D$ be a coalgebra map. The following
statements are equivalent:
\begin{enumerate}
\item[1)] $\varphi$ is a monomorphism in the category $k$-Coalg;
\item[2)] The functor $\varphi_{C,D}^{*}$ is full; \item[3)] $C
\square_{D} {\rm Ker} (\varphi) = 0$; \item[4)] The map $\eta_{C}
= \Delta_{C}: C \rightarrow C \square_{D} C$ is surjective (hence
is bijective); \item[5)] The unit of the adjunction
$(\varphi^{*}_{C,D},  - \square_{D} C)$, $\eta: 1_{{\mathcal
{M}}^{C}} \rightarrow (- \square_{D} C) \circ \varphi^{*}_{C,D}$
is a natural isomorphism; \item[6)] $H^{0}(N,C) = H^{0}(N,D)$, for
any $(C,C)$-bicomodule $N$; \item[7)] $\sum_{i \in
I}\varepsilon(a^{i})b^{i} = \sum_{i \in I} a^{i}
\varepsilon(b^{i})$, for all $\sum_{i \in I}a^{i} \otimes b^{i}
\in C \square_{D} C$.
\end{enumerate}
\end{theorem}

\begin{proof}
$1) \Rightarrow 6)$ Suppose $\varphi$ is a monomorphism of
coalgebras. It is easy to see that $H^{0}(N,C) \subset
H^{0}(N,D)$. Now let $\gamma \in H^{0}(N,D)$. Define the map
$\beta: C\rtimes N \rightarrow C$ by:
\begin{eqnarray*}
\beta(c,n) = c - n_{<-1>}\gamma(n_{<0>})+\gamma(n_{[0]})n_{[1]}
\end{eqnarray*}
 for all $(c,n) \in C\rtimes N$. We prove that $\beta$ is a coalgebra map. For all $(c,n) \in C\rtimes N$
 we have:
\begin{eqnarray*}
\varepsilon_{C} \circ \beta(c,n) &{=}& \varepsilon_{C}(c -
n_{<-1>}\gamma(n_{<0>})+\gamma(n_{[0]})n_{[1]})\\
&{=}& \varepsilon_{C}(c) -
\varepsilon_{C}(n_{<-1>})\gamma(n_{<0>})+\gamma(n_{[0]})
\varepsilon_{C}(n_{[1]})\\
&{=}& \varepsilon_{C}(c) -
\gamma\bigl(\varepsilon_{C}(n_{<-1>})n_{<0>}\bigl)+\gamma\bigl(n_{[0]}
\varepsilon_{C}(n_{[1]}\bigl)\\
&{=}& \varepsilon_{C}(c) - \gamma(n) + \gamma(n)\\
&{=}& \varepsilon(c,n)
\end{eqnarray*}
and
\begin{eqnarray*}
(\beta \otimes \beta)\circ \Delta(c,n) &{=}& (\beta \otimes
\beta)\bigl( (c_{(1)},0) \otimes (c_{(2)},0) + (m_{<-1>},0)
\otimes (0,m_{<0>}) + \\
&&+ (0,m_{[0]}) \otimes (m_{[1]},0)\bigl)\\
&{=}& \beta\bigl((c_{(1)},0)\bigl) \otimes
\beta\bigl((c_{(2)},0)\bigl) + \beta\bigl((m_{<-1>},0)\bigl)
\otimes \beta\bigl((0,m_{<0>})\bigl) +\\
&& + \beta\bigl((0,m_{[0]})\bigl) \otimes
\beta\bigl((m_{[1]},0)\bigl)\\
&{=}& c_{(1)} \otimes c_{(2)} + n_{<-1>} \otimes \bigl( -n_{<0>
<-1>} \gamma(n_{<0> <0>}) +\\
&& + \gamma(n_{<0> [0]})n_{<0> [1]}\bigl)+ \bigl(-n_{[0] <-1>}
\gamma(n_{[0] <0>}) + \gamma(n_{[0] [0]})n_{[0] [1]}\bigl) \otimes
n_{[1]}\\
&{=}& c_{(1)} \otimes c_{(2)} - n_{<-1>} \otimes n_{<0> <-1>}
\gamma(n_{<0> <0>}) + \gamma(n_{[0] [0]})n_{[0] [1]} \otimes
n_{[1]}\\
&{=}& c_{(1)} \otimes c_{(2)} - n_{<-1> (1)} \otimes n_{<0> (2)}
\gamma(n_{<0>}) + \gamma(n_{[0]})n_{[1] (1)} \otimes
n_{[1] (2)}\\
&{=}& c_{(1)} \otimes c_{(2)} -
\Delta_{C}\bigl(n_{<-1>}\gamma(n_{<0>})\bigl) +
\Delta_{C}\bigl(\gamma(n_{[0]})n_{[1]}\bigl)\\
&{=}& \Delta_{C}\bigl(c -
n_{<-1>}\gamma(n_{<0>})+\gamma(n_{[0]})n_{[1]}\bigl)\\
&{=}&\Delta_{C} \circ \beta(c,n)
\end{eqnarray*}
Hence, $\beta$ is a coalgebra map. Furthermore, it is easy to see
that $\varphi \circ \pi_{C} = \varphi \circ \beta$ and $\pi_{C} =
\beta$ because $\varphi$ is a monomorphism. Thus
$n_{<-1>}\gamma(n_{<0>}) = \gamma(n_{[0]})n_{[1]}$ which implies
that $\gamma \in H^{0}(N,C)$.\\

$6) \Rightarrow 7)$ $C \square_{D} C$ is a $(C,C)$-bicomodule with
left and right structures given by:\\
$$\psi_{C \square_{D} C}^{l}(\sum_{i \in I}a^{i} \otimes b^{i}) =
\sum_{i \in I} a^{i}_{(1)} \otimes \bigl(a^{i}_{(2)} \otimes
b^{i}\bigl)$$ and $$\psi_{C \square_{D} C}^{r}(\sum_{i \in I}a^{i}
\otimes b^{i}) = \sum_{i \in I}\bigl(a^{i} \otimes
b^{i}_{(1)}\bigl) \otimes b^{i}_{(2)}$$ for all $\sum_{i \in
I}a^{i} \otimes b^{i} \in C \square_{D} C$. We define the
$k$-linear map $T:C \square_{D} C \rightarrow k$ by: $$T(\sum_{i
\in I}a^{i} \otimes b^{i}) = \sum_{i \in I}\varepsilon(a^{i})
\varepsilon(b^{i})$$ for all $\sum_{i \in I}a^{i} \otimes b^{i}
\in C \square_{D} C$. Now let $\sum_{i \in I}a^{i} \otimes b^{i}
\in C \square_{D} C$, that is: $$\sum_{i \in I} a^{i}_{(1)}
\otimes \varphi(a^{i}_{(2)}) \otimes b^{i} = \sum_{i \in I} a^{i}
\otimes \varphi(b^{i}_{(1)}) \otimes b^{i}_{(2)}$$ By applying
$\varepsilon \otimes I \otimes \varepsilon$ in the above identity
we obtain $\sum_{i \in I} \varphi(a^{i}) \varepsilon(b^{i}) =
\sum_{i \in I} \varepsilon(a^{i}) \varphi(b^{i})$. Thus $T \in
H^{0}(C \square_{D} C,D) = H^{0}(C \square_{D} C,C)$
and it follows from here that $\sum_{i \in I} a^{i}\varepsilon(b^{i}) = \sum_{i \in I} \varepsilon(a^{i})b^{i}$.\\

$7) \Rightarrow 1)$ Let $M \in {\mathcal {M}}^{C}$. Define
$\nu_{M}: M \square_{D} C \rightarrow M$ by $\nu_{M}(\sum_{i \in
I} m^{i} \otimes c^{i}) = \sum_{i \in I} m^{i}\varepsilon(c^{i})$.
For any $\sum_{i \in I} m^{i} \otimes c^{i} \in M \square_{D}C$ we
have:
\begin{eqnarray*}
(\nu_{M} \circ I)\circ \rho_{M \square_{D}C}^{r}(\sum_{i \in I}
m^{i} \otimes c^{i}) &{=}& \sum_{i \in I}(\nu_{M}
\circ I)(m^{i} \otimes c^{i}_{(1)} \otimes c^{i}_{(2)})\\
&{=}& \sum_{i \in I} \nu_{M}(m^{i} \otimes c^{i}_{(1)}) \otimes c^{i}_{(2)}\\
&{=}& \sum_{i \in I} m^{i}\varepsilon(c^{i}_{(1)}) \otimes c^{i}_{(2)}\\
&{=}& \sum_{i \in I} m^{i}_{[0]}\varepsilon(m^{i}_{[1]}) \otimes c^{i}\\
&{=}& \sum_{i \in I} m^{i}_{[0]} \otimes \varepsilon(m^{i}_{[1]}) c^{i}\\
&{=}& \sum_{i \in I} m^{i}_{[0]} \otimes m^{i}_{[1]} \varepsilon(c^{i})\\
&{=}& \sum_{i \in I} \rho_{M}^{r}(m^{i} \varepsilon(c^{i}))\\
&{=}& (\rho_{M}^{r} \circ \nu_{M})(\sum_{i \in I} m^{i} \otimes
c^{i})
\end{eqnarray*}
where we use the fact that $\sum_{i \in I} m^{i}_{[1]} \otimes
c^{i} \in C \square_{D}C$ for all $\sum_{i \in I} m^{i} \otimes
c^{i} \in M \square_{D}C$. Thus $\nu_{M}$ is a morphism of right
$C$-comodules. Moreover, in the computations above we also prove
that $(\rho_{M}^{r} \circ \nu_{M})(\sum_{i \in I} m^{i} \otimes
c^{i}) = \sum_{i \in I} m^{i} \otimes c^{i}$ for all $\sum_{i \in
I} m^{i} \otimes c^{i} \in M \square_{D}C$. It follows that for
all $M \in {\mathcal {M}}^{C}$ there exists a morphism of right
$C$-comodules such that $\rho_{M}^{r} \circ \nu_{M} = I$. Having
in mind that $\eta_{M} = \rho_{M}^{r}$, as remarked before, THIS
is enough to prove that $\varphi_{C,D}^{*}$ is a full functor.
Now, in the light of \cite[Theorem 3.5]{N} we obtain that
$\varphi$ is a monomorphism in $k$-Coalg.
\end{proof}

In view of a remark of A. Chirv\v asitu (\cite{C}) that a morphism
of Hopf algebras is a monomorphism if and only if it is a
monomorphism viewed as a morphism of coalgebras we obtain the
following useful fact:

\begin{corollary}
Let $\varphi: K \rightarrow L$ be a Hopf algebra map. The
following are equivalent:
\begin{enumerate}
\item[1)] $\varphi: K \rightarrow L$ is a Hopf algebra
monomorphism; \item[2)] The map $\eta_{K} = \Delta_{K}: K
\rightarrow K \square_{L} K$ is surjective (hence is bijective);
\item[3)] $H^{0}(N,K) = H^{0}(N,L)$ for any $(K,K)$-bicomodule
$N$; \item[4)] $\sum_{i \in I}\varepsilon(x^{i})y^{i} = \sum_{i
\in I} x^{i} \varepsilon(y^{i})$ for all $\sum_{i \in I}x^{i}
\otimes y^{i} \in K \square_{L} K$.
\end{enumerate}
\end{corollary}

\begin{example}
Let $\pi: \mathcal{M}^{2}(k) \rightarrow \mathcal{M}^{2}(k)/I$ be
the canonical projection, where $k$ is a field and $I$ is the
coideal of the comatrix coalgebra $\mathcal{M}^{2}(k)$ generated
by the elements $c_{21}$. It is proved in \cite{N}, using a result
on epimorphisms of finite dimensional algebras \cite[Theorem
3.2]{N}, that $\pi$ is a non-injective monomorphism of coalgebras.
However, this can be easily shown by a simple computation using
$7)$ of \thref{c}.
\end{example}

\begin{remark}
The equivalence of the statements $(1)$ and $(4)$ in \thref{c} can
be alternatively proved by applying the categorical duality $(k$ -
Coalg$)^{op} \cong \mathcal{P}\mathcal{C}_{k}$, $C \mapsto C^{*}$,
between the category $k - $Coalg of $k$-coalgebras and the
category $\mathcal{P}\mathcal{C}_{k}$ of pseudocompact
$k$-algebras described in \cite[Section 3]{Sim}. One should apply
the isomorphism $(C \square C)^{*} \cong C^{*} \widehat{\otimes}
C^{*}$ and the results of Knight (\cite{KNI}), where
$\widehat{\otimes}$ is the complete tensor product, see the
monograph \cite{DNR}.
\end{remark}

\section*{Acknowledgement}

The author wishes to thank Professor Gigel Militaru, who suggested
the problem studied here, as well as colleagues Alexandru Chirv\v
asitu and Drago\c{s} Fr\v a\c{t}il\v a for stimulating
discussions. Also the author is grateful to the referee for many
valuable suggestions which improved the first version of the
paper.

\end{document}